\documentclass[a4paper,12pt]{amsart}
\usepackage[margin=3 cm]{geometry}
\usepackage[OT2,OT1]{fontenc}
\newcommand\cyr
{
\renewcommand\rmdefault{wncyr}
\renewcommand\sfdefault{wncyss}
\renewcommand\encodingdefault{OT2}
\normalfont
\selectfont
}
\DeclareTextFontCommand{\textcyr}{\cyr}
\def\cprime{\char"7E }

\usepackage{fullpage}
\newtheorem{theorem}{Theorem}[section]

\newcommand{\ovprt}{\overline{\partial}}
\newcommand{\ovli}{\overline}

\numberwithin{equation}{section}

\title{ Pauli operators and the $\overline\partial$-Neumann problem}

\author{ Friedrich Haslinger}

\thanks{Partially supported by the FWF-grant  P28154.}

 \address{ F. Haslinger: Fakult\"at  f\"ur Mathematik, Universit\"at Wien,
Oskar-Morgenstern-Platz 1, A-1090 Wien, Austria}
\email{ friedrich.haslinger@univie.ac.at}


\begin{document}

\maketitle
\renewcommand\abstractname{\hspace*{-0.5cm}{\cyr \bf ANNOTATSIYA}}
\begin{abstract} 
{\cyr { Metodami kompleksnogo analiza (v chastnosti, posredstvom primeneniya $\ovprt$-operatora Ne{\u i}mana) rassmatrivaet \hspace{-0.15cm}s{ya}
 spektral\cprime ny{\u i} analiz operatorov Pauli. }}

\end{abstract}
\renewcommand\abstractname{\bf Abstract}
\begin{abstract} We apply methods from complex analysis, in particular the
$\ovprt$-Neumann operator, to investigate spectral properties of Pauli operators.

\end{abstract}

\vskip 1.5 cm

\section{Introduction}
\vskip 1cm

Let $\varphi : \mathbb R^{2n} \longrightarrow \mathbb R$ be a $\mathcal C^2$-function. We consider the Schr\"odinger operators with magnetic field of the form
$$P_\pm = -\Delta_A \pm V,$$
also called Pauli operators, where 
$$A=\frac{1}{2} \left (-\frac{\partial \varphi}{\partial y_1} , \frac{\partial \varphi}{\partial x_1}, \dots, 
-\frac{\partial \varphi}{\partial y_n}, \frac{\partial \varphi}{\partial x_n} \right )$$
 is the magnetic potential 
 and
$$\Delta_A = \sum_{j=1}^n \left [ \left ( -\frac{\partial }{\partial x_j} - \frac{i}{2} \frac{\partial \varphi}{\partial y_j} \right )^2 +   \left ( -\frac{\partial }{\partial y_j} + \frac{i}{2} \frac{\partial \varphi}{\partial x_j} \right )^2 \right ],$$
and $V= \frac{1}{2} \Delta \varphi;$ we wrote elements of $\mathbb R^{2n}$ in the form $(x_1,y_1, \dots, x_n,y_n);$ we will identify $\mathbb R^{2n}$ with $\mathbb C^n,$ writing $(z_1, \dots ,z_n)= (x_1,y_1, \dots, x_n,y_n),$ this is mainly because we will use methods of complex analysis to analyze spectral properties of the above Schr\"odinger operators with magnetic field.  

For $n=1,$ there is an interesting connection to Dirac and Pauli operators: recall  the definition of $A$ in this case and define the Dirac operator $\mathcal{D}$  by 

\begin{equation}\label{eq: schr18}
\mathcal{D}= (-i \frac{\partial}{\partial x}-A_1)\, \sigma_1 + (-i \frac{\partial}{\partial y}-A_2)\, \sigma_2 = \mathcal{A}_1 \sigma_1 + \mathcal{A}_2 \sigma_2,
\end{equation}
where
$$
\sigma_1=\left(
\begin{array}{cc}
 0&1\\1&0
\end{array}\right)\;,\;\sigma_2=\left(
\begin{array}{cc}
 0&-i\\i&0
\end{array}\right)\;.
$$
Hence we can write
$$
\mathcal{D} =
\left(
\begin{array}{cc}
 0&\mathcal{A}_1-i\mathcal{A}_2\\
 \mathcal{A}_1+i \mathcal{A}_2&0
\end{array}\right) .
$$
We remark that $i(\mathcal{A}_2\mathcal{A}_1 - \mathcal{A}_1\mathcal{A}_2)=B$ and hence
it turns out that the square of $\mathcal D$ is diagonal with the Pauli operators $P_{\pm }$ on the diagonal:
\begin{eqnarray*}\label{eq: schr19}
\mathcal D^2 &=& 
\left(
\begin{array}{cc}
 \mathcal{A}_1^2 - i(\mathcal{A}_2\mathcal{A}_1 - \mathcal{A}_1\mathcal{A}_2) + \mathcal{A}_2^2&0\\0& \mathcal{A}_1^2 + i(\mathcal{A}_2\mathcal{A}_1 - \mathcal{A}_1\mathcal{A}_2) + \mathcal{A}_2^2
\end{array}\right)\\
&=&\left(
\begin{array}{cc}
 P_-&0\\0&P_+
\end{array}\right),
\end{eqnarray*}
where 
\begin{equation*}\label{eq: plusminus}
P_{\pm}= 
 \left ( -i \, \frac{\partial}{\partial x}- A_1 \right )^2
 +  \left ( -i \, \frac{\partial}{\partial y}- A_2 \right )^2 \pm
B\, = -\Delta_A \pm B,
\end{equation*}
see \cite{CFKS} and \cite{Th}.

\vskip 0.5 cm

Our aim is to investigate spectral properties of the Pauli operators $P_\pm.$
For this purpose we will use methods from complex analysis, the weighted $\overline \partial$-complex. We suppose that $\varphi : \mathbb C^n \longrightarrow \mathbb R$ is a plurisubharmonic $\mathcal C^2$-function.

Let 
$$L^2( \mathbb{C}^n, e^{-\varphi}) = \{ g: \mathbb{C}^n \longrightarrow \mathbb{C} \ {\text{measurable}} \,:
\|g\|^2_\varphi =(g,g)_\varphi = \int_{\mathbb{C}^n} |g|^2 e^{-\varphi}\, d\lambda < \infty \}.$$
Let $1\le q \le n$ and
$$f= \sum_{|J|=q}\, ' \, f_J\,d\overline z_J ,$$
where the sum is taken only over increasing multiindices $J=(j_1, \dots , j_q)$ and $d\overline z_J = d\overline z_{j_1} \wedge \dots \wedge d\overline z_{j_q}$ and $f_J\in L^2(\mathbb C^n, e^{-\varphi}).$

We write $f\in L^2_{(0,q)}(\mathbb C^n, e^{-\varphi})$ and define 
$$\ovprt f = \sum_{|J|=q}\, ' \, \sum_{j=1}^n \frac{\partial f_J}{\partial \overline z_j}\, d\overline z_j \wedge d\overline z_J$$
for $1\le q \le n-1$ and 
$${\text{dom}}(\ovprt) = \{ f \in L^2_{(0,q)}(\mathbb C^n, e^{-\varphi}) \, : \, \ovprt f \in L^2_{(0,q+1)}(\mathbb C^n, e^{-\varphi}) \} ,$$
where the derivatives are taken in the sense of distributions.

In this way $\ovprt$ becomes a densely defined closed operator and its adjoint $\ovprt^*_\varphi$ depends on the weight $\varphi.$

We consider the weighted $\ovprt$-complex
\begin{equation*}
L^2_{(0,q-1)}(\mathbb{C}^n , e^{-\varphi} )\underset{\underset{\ovprt_\varphi^* }
\longleftarrow}{\overset{\ovprt }
{\longrightarrow}} L^2_{(0,q)}(\mathbb{C}^n , e^{-\varphi} )
\underset{\underset{\ovprt_\varphi^* }
\longleftarrow}{\overset{\ovprt }
{\longrightarrow}} L^2_{(0,q+1)}(\mathbb{C}^n , e^{-\varphi} )
\end{equation*}
and we set 
$$\Box_{\varphi}^{(0,q)}= \ovprt\, \ovprt_\varphi^* +	 \ovprt_\varphi^* \ovprt,$$
where
$${\text{dom}}(\Box_\varphi^{(0,q)}) = \{ u \in {\text{dom}}(\ovprt) \cap {\text{dom}}(\ovprt^*_\varphi): \ovprt u \in  {\text{dom}}(\ovprt^*_\varphi), \ovprt^*_\varphi u \in  {\text{dom}}(\ovprt)\}.$$

It turns out that $\Box_{\varphi}^{(0,q)}$ is a densely defined, non-negative self-adjoint operator, which has a uniquely determined self-adjoint square root $(\Box_{\varphi}^{(0,q)})^{1/2}.$ The domain of
$(\Box_{\varphi}^{(0,q)})^{1/2})$ coincides with ${\text{dom}}(\ovprt) \cap {\text{dom}}(\ovprt^*_\varphi),$ which is also the domain of the corresponding quadratic form
$$Q_\varphi (u,v):=(\ovprt u, \ovprt v)_\varphi + (\ovprt^*_\varphi u, \ovprt^*_\varphi v)_\varphi,$$
and ${\text{dom}}(\Box_\varphi^{(0,q)})$ is a core of $(\Box_{\varphi}^{(0,q)})^{1/2},$ see for instance \cite{Dav}. 

Next we consider the Levi matrix 
$$M_\varphi = \left ( \frac{\partial^2 \varphi}{\partial z_j \partial \overline z_k} \right )_{j,k=1}^n$$

and suppose that the lowest eigenvalue $\mu_\varphi$ of $M_\varphi$ satisfies
\begin{equation} \label{perss}
\liminf_{|z| \to \infty} \mu_\varphi (z)>0.
\end{equation}

\eqref{perss} implies that $\Box_\varphi ^{(0,1)}$ is injective and that the bottom of the essential spectrum
$\sigma_e(\Box_\varphi ^{(0,1)})$ is positive (Persson's Theorem), see \cite{HaHe}.
Now it follows that $\Box_\varphi ^{(0,1)}$ has a bounded inverse,  which we denote by 
$$N_\varphi ^{(0,1)}: L^2_{(0,1)}(\mathbb{C}^n , e^{-\varphi} ) \longrightarrow L^2_{(0,1)}(\mathbb{C}^n , e^{-\varphi} ).$$
Using the square root of  $N_\varphi ^{(0,1)}$ we get the basic estimates
\begin{equation}\label{coerc}
\|u\|^2_\varphi \le C ( \|\ovprt u \|^2_\varphi + \|\ovprt^*_\varphi  u\|^2_\varphi ),
\end{equation}
for all $u\in {\text{dom}}(\ovprt) \cap {\text{dom}}(\ovprt^*_\varphi),$ see \cite{Has10} for more details.

In the following it will be important to know conditions on $\varphi$ implying that
the Bergman space of entire functions
$$A^2(\mathbb C^n, e^{-\varphi}) := L^2(\mathbb C^n, e^{-\varphi}) \cap \mathcal O (\mathbb C^n)$$
is infinite dimensional. This space coincides with ${\text{ker}} \ovprt,$ where
$$\ovprt : L^2(\mathbb C^n, e^{-\varphi}) \longrightarrow L^2_{(0,1)}(\mathbb C^n, e^{-\varphi}).$$ 
If $n=1,$ we can use the following concept: let $D(z,r)=\{w : |z-w|<r \};$
a non-negative Borel measure $\mu$ on $\mathbb{C}$ is doubling\index{doubling measure}, if there exists a constant $C>0$ such that for any $z\in \mathbb{C}$ and any $r>0$
\begin{equation}\label{doub}
\mu(D(z,r)) \le C \mu(D(z,r/2)).
\end{equation}

It can be shown that
\begin{equation}\label{doub2}
\mu(D(z,2r)) \ge (1+ C^{-3}) \mu(D(z,r)),
\end{equation}
for each $z\in \mathbb{C}$ and for each $r>0$; in particular $\mu(\mathbb{C})=\infty,$ unless $\mu (\mathbb{C})=0$ (see \cite{St}).

Example:
if $p(z, \overline z)$ is a polynomial on $\mathbb{C}$ of degree $d,$ then
$$d\mu (z)= |p(z, \overline z)|^{a}\, d\lambda (z), \ a> -\frac{1}{d}$$
is a doubling measure on $\mathbb{C},$ see \cite{St}.

\begin{theorem}\label{s: infdim}\cite{Ch}, \cite{RS}
Let $\varphi: \mathbb{C} \longrightarrow \mathbb{R}_{+}$ be a subharmonic $\mathcal{C}^{2}$-function.
Suppose that $d\mu=\triangle \varphi\, d\lambda$ is a non-trivial doubling measure.

Then the weighted space of entire functions
$$A^2( \mathbb{C}, e^{-\varphi}) =\{ f \, {\text{entire}} : \| f\|^2_{\varphi}= \int_{\mathbb{C}} |f|^2 e^{-\varphi}\,d\lambda < \infty \}$$
 is of infinite dimension.
 \end{theorem}

More general, in $\mathbb C^n,$ H\"ormanders $L^2$-estimates for the solution of the inhomogeneous Cauchy-Riemann equations yield

\begin{theorem}\label{berinf}\cite{Shi}, \cite{Has10}
Suppose that the lowest eigenvalue $\mu_\varphi$ satisfies 
\begin{equation} \label{berinf1}
\lim_{|z| \to \infty} |z|^2 \mu_\varphi (z)= +\infty.
\end{equation}
Then the weighted space of entire functions
$$A^2( \mathbb{C}^n, e^{-\varphi}) =\{ f \, {\text{entire}} : \| f\|^2_{\varphi}= \int_{\mathbb{C}^n} |f|^2 e^{-\varphi}\,d\lambda < \infty \}$$
 is of infinite dimension.
 \end{theorem}
 
 Concerning compactness of the $\ovprt$-Neumann operator we have the following result:
 
 \begin{theorem}\label{compact3}\cite{Has10}
 Let $1\le q \le n.$ Suppose that the sum $s_q$ of the smallest $q$ eigenvalues of the Levi matrix $M_\varphi$ satisfies
 \begin{equation}\label{compact4}
 \lim_{|z|\to \infty} s_q(z)=+\infty.
 \end{equation}
 Then the $\ovprt$-Neumann operator 
 $$N_\varphi^{(0,q)} : L^2_{(0,q)}(\mathbb C^n, e^{-\varphi}) \longrightarrow
 L^2_{(0,q)}(\mathbb C^n, e^{-\varphi})$$
 is compact.
 \end{theorem}
 
 The next result asserts that compactness percolates up the $\ovprt$-complex.
 
 \begin{theorem}\label{percol}\cite{Has10}
Let $1\le q \le n-1.$ Suppose that $N_\varphi^{(0,q)}$ is compact. Then $N_\varphi^{(0,q+1)}$ is also compact.
 \end{theorem}
\vskip 0.5 cm 
 We will also consider special weight functions, the so-called decoupled weights,
 and, using the tensor product structure of the essential spectrum $\sigma_e(\Box_\varphi^{(0,q)})$ we get the following (see \cite{BeHa})
\vskip 0.5 cm
 
 \begin{theorem}\label{decou}
 Let $\varphi_j \in \mathcal C^2(\mathbb C, \mathbb R)$ for $1\le j \le n$ with $n\ge 2,$ and set 
 $$\varphi (z_1,\dots , z_n) := \varphi_1(z_1) + \dots + \varphi_n(z_n).$$
 Assume that all $\varphi_j$ are subharmonic and such that $\Delta \varphi_j$ defines a nontrivial doubling measure. Then
 
 (i) dim$({\text{ker}}(\Box_\varphi^{(0,0)}) = {\text{dim}}(A^2(\mathbb C^n, e^{-\varphi}))= \infty,$ where $\Box_\varphi^{(0,0)}= \ovprt^*_\varphi \, \ovprt,$
 
 (ii) ${\text{ker}}(\Box_\varphi^{(0,q)})=\{ 0 \},$ for $q\ge 1,$
 
 (iii) $N_\varphi^{(0,q)}$ is bounded for $0\le q \le n,$
 
 (iv) $N_\varphi^{(0,q)}$ with $0\le q \le n-1$ is not compact, and 
 
 (v) $N_\varphi^{(0,n)}=\ovprt \, \ovprt^*_\varphi$ is compact if and only if
 
 $$\lim_{|z| \to \infty} \int_{B_1(z)} {\text{tr}}(M_\varphi)\, d\lambda = \infty,$$
 where $B_1(z) =\{ w \in \mathbb C^n : |w-z|< 1\}.$
 
 \end{theorem}
\vskip 1 cm 
\section{Pauli operators}
\vskip 1cm

Now we apply the results on the weighted $\ovprt$-Neumann operator to derive spectral properties of the Pauli operators and discuss some special examples.
 
\begin{theorem}\label{pauli3}
Let $\varphi : \mathbb C^n \longrightarrow \mathbb R$ be a plurisubharmonic $\mathcal C^2$-function. Suppose that the smallest eigenvalue $\mu_\varphi$ of the Levi matrix $M_\varphi$ satisfies
\begin{equation}\label{comppauli}
\lim_{|z|\to \infty}\mu_\varphi (z) = \infty.
\end{equation}
Let 
$$A=\frac{1}{2} \left (-\frac{\partial \varphi}{\partial y_1} , \frac{\partial \varphi}{\partial x_1}, \dots, 
-\frac{\partial \varphi}{\partial y_n}, \frac{\partial \varphi}{\partial x_n} \right )$$
and $V=\frac{1}{2} \Delta \varphi.$
Then the Pauli operator $P_-= -\Delta_A -V$ fails to have a compact resolvent, whereas the Pauli operator  $P_+ = -\Delta_A + V$ has a compact inverse operator acting on $L^2(\mathbb R^{2n}).$

\end{theorem}

\begin{proof} For the proof we first consider the complex Laplacian $\Box_\varphi^{(0,0)}= \ovprt^*_\varphi \, \ovprt,$ which acts on $L^2(\mathbb C^n, e^{-\varphi})$ at the beginning of the weighted $\ovprt$-complex as a non-negative self-adjoint, densely defined operator, we take the maximal extension from $\mathcal C^\infty_0 (\mathbb C^n),$ as  $\Box_\varphi^{(0,0)}$ is essentially self-adjoint, there is only one self-adjoint extension. For $f\in \mathcal C^\infty_0 (\mathbb C^n)$ we get
$$\Box_\varphi^{(0,0)} f= \ovprt^*_\varphi \, \ovprt f = - \sum_{j=1}^n \left (
\frac{\partial}{\partial z_j} - \frac{\partial \varphi}{\partial z_j} \right ) \frac{\partial f}{\partial \overline z_j}.$$
Now we apply the isometry
$$U_{\varphi} : L^2(\mathbb C^n) \longrightarrow L^2(\mathbb C^n, , e^{-\varphi})$$
defined by $U_{\varphi} (g)= e^{\varphi /2} g,$ for $g\in L^2(\mathbb C^n),$
and afterwards the isometry
$$U_{-\varphi} : L^2(\mathbb C^n, e^{-\varphi}) \longrightarrow L^2(\mathbb C^n)$$
defined by $U_{-\varphi} (f)= e^{-\varphi /2} f,$ for $f\in L^2(\mathbb C^n, e^{-\varphi}).$ 
Hence we get
\begin{align*}
e^{-\varphi/2} \Box_\varphi^{(0,0)}(e^{\varphi/2}g) \\
&= \sum_{j=1}^n \left ( -\frac{\partial^2 g}{\partial z_j \partial \overline z_j} + \frac{1}{2} \frac{\partial \varphi}{\partial z_j}\frac{\partial g}{\partial \overline z_j}
 -\frac{1}{2} \frac{\partial \varphi}{\partial \overline z_j}\frac{\partial g}{\partial  z_j} + \frac{1}{4}  \frac{\partial \varphi}{\partial z_j}\frac{\partial \varphi}{\partial \overline z_j} -\frac{1}{2}\frac{\partial^2 \varphi}{\partial z_j \partial \overline z_j}g
 \right ),
\end{align*}
and separating into real and imaginary part 
$$\frac{\partial}{\partial z_j}= \frac{1}{2} \left ( \frac{\partial}{\partial x_j} - 
\frac{\partial}{\partial y_j} \right ) \ , \ 
\frac{\partial}{\partial \overline z_j}= \frac{1}{2} \left ( \frac{\partial}{\partial x_j} + 
\frac{\partial}{\partial y_j} \right )
$$
we obtain 
\begin{equation}\label{pauli10}
e^{-\varphi/2} \Box_\varphi^{(0,0)}(e^{\varphi/2}g) = \frac{1}{4} (-\Delta_A-V)g,
\end{equation}
where 
$$A=\frac{1}{2} \left (-\frac{\partial \varphi}{\partial y_1} , \frac{\partial \varphi}{\partial x_1}, \dots, 
-\frac{\partial \varphi}{\partial y_n}, \frac{\partial \varphi}{\partial x_n} \right )$$
and 
 $$V= 2 {\text{tr}}(M_\varphi)=\frac{1}{2} \Delta \varphi.$$
Since the kernel of $\ovprt : L^2(\mathbb C^n, e^{-\varphi}) \longrightarrow
L^2_{(0,1)}(\mathbb C^n, e^{-\varphi})$ coincides with the Bergman space
$A^2(\mathbb C^n, e^{-\varphi})$ we get from \eqref{pauli10} and the fact that 
\eqref{comppauli} implies that $A^2(\mathbb C^n, e^{-\varphi})$ is infinite dimensional (see Theorem \ref{berinf}) that $0 \in \sigma_e(\Box_\varphi^{(0,0)}).$ Hence $\Box_\varphi^{(0,0)}$ fails to be with compact resolvent.

In order to show that the Pauli operator $P_+$ has a compact inverse we look at the end of the weighted $\ovprt$-complex. 

Let $u=u\, d\ovli z_1\wedge \dots \wedge d\ovli z_n$ be a smooth $(0,n)$-form belonging to the domain of $\Box_\varphi^{(0,n)}.$ For $1 \leq j \leq n$ denote by $K_j$ the increasing multiindex $K_j := (1,\dots,j-1,j+1,\dots,n)$ of length $n-1$.\
Then
$$ \ovprt_\varphi^* u = \sum_{j=1}^n (-1)^{j+1} \bigg(\frac{\partial \varphi}{\partial z_j}u-\frac{\partial u}{\partial z_j}\bigg)\,d\ovli z_{K_j}.$$
Hence 
\begin{align*}
 \ovprt \ovprt_\varphi^* u
  &= \Bigg[\sum_{j=1}^n \frac{\partial }{\partial \ovli z_j} \bigg(\frac{\partial \varphi}{\partial z_j}u-\frac{\partial u}{\partial z_j}\bigg)\Bigg] d\ovli z_1 \wedge \dots \wedge d\ovli z_n\\
  &= \Bigg[\sum_{j=1}^n  \bigg(\frac{\partial^2\varphi}{\partial z_j\partial \ovli z_j}u+
      \frac{\partial \varphi}{\partial z_j}
      \frac{\partial u}{\partial \ovli z_j} -
      \frac{\partial^2 u}{\partial z_j\partial \ovli z_j}
      \bigg)\Bigg] d\ovli z_1 \wedge \dots \wedge d\ovli z_n.
\end{align*}
Conjugation with the unitary operator $U_{-\varphi}: L^2(\mathbb C^n,e^{-\varphi}) \to L^2(\mathbb C^n)$ of multiplication by $e^{-\varphi/2}$ gives
\begin{equation*}\label{eq:complex_laplacian_smooth_n1}
e^{-\varphi/2}  \Box_\varphi^{(0,n)} e^{\varphi/2}g= \sum_{j=1}^n \bigg(
-\frac{\partial^2 g}{\partial z_j\partial \ovli z_j}-\frac{1}{2} \frac{\partial \varphi}{\partial \ovli z_j}
\frac{\partial g}{\partial z_j} + \frac{1}{2} \frac{\partial \varphi}{\partial z_j}
\frac{\partial g}{\partial \ovli z_j} + \frac{1}{4} \frac{\partial \varphi}{\partial \ovli z_j}
\frac{\partial \varphi}{\partial z_j} g + \frac{1}{2}\frac{\partial^2 \varphi}{\partial z_j\partial \ovli z_j}g\bigg),
\end{equation*}
where $g\in L^2(\mathbb C^n)$ and we just wrote down the coefficient of the corresponding $(0,n)$-form.
This operator can be expressed by real variables in the form
\begin{equation}\label{eq:complex_laplacian_smooth_n2}
 e^{-\varphi/2}  \Box_\varphi^{(0,n)} e^{\varphi/2}g = \frac{1}{4} (-\Delta_A +V)g,
\end{equation}
with
$$ \Delta_A= \sum_{j=1}^n \Bigg[\bigg( -\frac{\partial}{\partial x_j} - \frac{i}{2} \frac{\partial \varphi}{\partial y_j} \bigg)^2
+  \bigg( -\frac{\partial}{\partial y_j} + \frac{i}{2} \frac{\partial \varphi}{\partial x_j} \bigg)^2 \Bigg], $$
and $V = 2{\text{tr}}M_\varphi)$.
It follows that $-\Delta_A +V$ is a Schr\"odinger operator on $L^2(\mathbb R^{2n})$ with magnetic vector potential
$$ A =\frac{1}{2} \bigg(-\frac{\partial \varphi}{\partial y_1}, \frac{\partial \varphi}{\partial x_1}, \dotsc, -\frac{\partial \varphi}{\partial y_n}, \frac{\partial \varphi}{\partial x_n}\bigg), $$
where $z_j=x_j+iy_j$, $j=1,\dots,n$, and non-negative electric potential $V$ in the case where $\varphi$ is plurisubharmonic.

From \eqref{comppauli} we get that $N_\varphi^{(0,1)}$ is compact (Theorem \ref{compact3}) and by Theorem \ref{percol} that $N_\varphi^{(0,n)}$ is compact.
Finally \eqref{eq:complex_laplacian_smooth_n2} implies that the Pauli operator $P_+$ has a compact inverse.

\end{proof}

For decoupled  weights $\varphi (z_1, \dots, z_n)= \varphi_1(z_1) + \dots +\varphi_n(z_n)$ even more can be said.

\begin{theorem}\label{decou1}
 Let $\varphi_j \in \mathcal C^2(\mathbb C, \mathbb R)$ for $1\le j \le n$ with $n\ge 1,$ and set 
 $$\varphi (z_1,\dots , z_n) := \varphi_1(z_1) + \dots + \varphi_n(z_n).$$
 Assume that all $\varphi_j$ are subharmonic and such that $\Delta \varphi_j$ defines a nontrivial doubling measure. 
 
 Let 
$$A=\frac{1}{2} \left (-\frac{\partial \varphi}{\partial y_1} , \frac{\partial \varphi}{\partial x_1}, \dots, 
-\frac{\partial \varphi}{\partial y_n}, \frac{\partial \varphi}{\partial x_n} \right )$$
and $V=\frac{1}{2} \Delta \varphi.$
Then the Pauli operator $P_-= -\Delta_A -V$ fails to have a compact resolvent, 
the Pauli operator $P_+= -\Delta_A+V$ has a compact inverse if and only if 
$$\lim_{|z| \to \infty} \int_{B_1(z)} {\text{tr}}(M_\varphi)\, d\lambda = \infty,$$
where $B_1(z) =\{ w \in \mathbb C^n : |w-z|< 1\}.$

\end{theorem}
\begin{proof}
By Theorem \ref{s: infdim} we obtain that $A^2( \mathbb{C}^n, e^{-\varphi})$ is infinite dimensional. So, $P_-$ fails to be with compact resolvent. The assertion about $P_+$ follows from Theorem \ref{decou}.
\end{proof}

\vskip 0.5 cm
{\bf Example:} For $\varphi (z_1, \dots, z_n)= |z_1|^2 + \dots + |z_n|^2$ both
Pauli operators $P_-$ and $P_+$ fail to be with compact resolvent.

\vskip 0.3 cm
Finally, we get the following result  for the Dirac operators \eqref{eq: schr18}.

\begin{theorem}\label{dir1}
Let $n=1$ and let $\varphi $ be a subharmonic $\mathcal C^2$-function such that $\Delta \varphi $ defines a nontrivial doubling measure. Then the Dirac operator
$$\mathcal{D}=\left (-i \frac{\partial}{\partial x}+ \frac{1}{2} \frac{\partial \varphi}{\partial y}\right )\, \sigma_1 + \left (-i \frac{\partial}{\partial y}-\frac{1}{2} \frac{\partial \varphi}{\partial x}\right )\, \sigma_2 ,$$
where 
$$
\sigma_1=\left(
\begin{array}{cc}
 0&1\\1&0
\end{array}\right)\;,\;\sigma_2=\left(
\begin{array}{cc}
 0&-i\\i&0
\end{array}\right)\;,
$$
fails to be with compact resolvent.
\end{theorem}

\begin{proof}By spectral analysis (see \cite{Has10}) it follows that $\mathcal D^2$ has compact resolvent, if and only if $\mathcal D$ has compact resolvent.
Suppose that $\mathcal D$ has compact resolvent. Since
\begin{eqnarray*}\label{eq: schr19}
\mathcal D^2 
&=&\left(
\begin{array}{cc}
 P_-&0\\0&P_+
\end{array}\right),
\end{eqnarray*}
this would imply that both Pauli operators $P_-$ and $P_+$ have compact resolvent, contradicting Theorem \ref{decou1}.

\end{proof}

\vskip 1 cm

\bibliographystyle{amsplain}
\bibliography{mybibliography}

\providecommand{\bysame}{\leavevmode\hbox to3em{\hrulefill}\thinspace}
\providecommand{\MR}{\relax\ifhmode\unskip\space\fi MR }
\providecommand{\MRhref}[2]{%
  \href{http://www.ams.org/mathscinet-getitem?mr=#1}{#2}
}
\providecommand{\href}[2]{#2}
\begin{thebibliography}{10}

\bibitem{BeHa}
F.~Berger and F.~Haslinger, \emph{{On some spectral properties of the weighted
  $\ovprt$-Neumann operator}}, J. Math. Kyoto Univ., to appear, arXiv (2016),
  1509.08741.

\bibitem{Ch}
M.~Christ, \emph{{On the $\ovprt $ equation in weighted $L^2$ norms in $\mathbb
  C^1$ }}, J. of Geometric Analysis \textbf{1} (1991), 193--230.

\bibitem{CFKS}
H.L. Cycon, R.G. Froese, W.~Kirsch, and B.~Simon, \emph{{ Schr\"odinger
  Operators with Applications to Quantum Mechanics and Global Geometry }}, Text
  and Mongraphs in Physics, Springer - Verlag, 1987.

\bibitem{Dav}
E.B. Davies, \emph{Spectral theory and differential operators}, Cambridge
  studies in advanced mathematics, vol.~42, Cambridge University Press,
  Cambridge, 1995.

\bibitem{Has10}
F.~Haslinger, \emph{{The $\ovprt$-Neumann problem and Schr\"odinger
  operators}}, de Gruyter Expositions in Mathematics 59, Walter De Gruyter,
  2014.

\bibitem{HaHe}
F.~Haslinger and B.~Helffer, \emph{{ Compactness of the solution operator to
  $\ovprt $ in weighted $L^ 2$ - spaces}}, J. of Functional Analysis
  \textbf{243} (2007), 679--697.

\bibitem{RS}
G.~Rozenblum and N.~Shirokov, \emph{{Infiniteness of zero modes for the Pauli
  operator with singular magnetic field}}, J. Func. Anal. \textbf{233} (2006),
  135--172.

\bibitem{Shi}
I.~Shigekawa, \emph{{Spectral properties of Schr\"odinger operators with
  magnetic fields for a spin $ \ 1/2 \ $ particle}}, J. of Functional Analysis
  \textbf{101} (1991), 255--285.

\bibitem{St}
E.~Stein, \emph{Harmonic analysis: Real-variable methods, orthogonality, and
  oscillatory integrals}, Princeton University Press, Princeton, New Jersey,
  1993.

\bibitem{Th}
B.~Thaller, \emph{{The Dirac Equation}}, Texts Monogr. Phys., Springer Verlag,
  1991.

\end{thebibliography}

\end{document}